\documentclass[11pt]{article}
\usepackage{amssymb}
\usepackage{latexsym}
\usepackage{amsthm}
\usepackage{amscd}
\usepackage{amsmath}
\usepackage{diagrams}
\usepackage{indentfirst}
\theoremstyle{definition}
\newtheorem{thm}{Theorem}[section]
\newtheorem{lem}[thm]{Lemma}
\newtheorem{th-def}[thm]{Theorem-Definition}
\newtheorem{cor}[thm]{Corollary}
\newtheorem{defn-lem}[thm]{Definition-Lemma}

\newtheorem{prop}[thm]{Proposition}

\newtheorem{rem}[thm]{Remark}

\numberwithin{equation}{section}

\allowdisplaybreaks[4]

\def \Q{{\mathbb Q}}
\def \N{{\mathbb N}}
\def \C{{\mathbb C}}
\def \F{{\mathbb F}}

\def \Z{{\mathbb Z}}
\def \R{{\mathbb R}}

\def\map#1.#2.{#1 \longrightarrow #2}
\def\rmap#1.#2.{#1 \dasharrow #2}

\DeclareMathOperator{\Hom}{Hom}
\DeclareMathOperator{\Spf}{Spf}

\DeclareMathOperator{\im}{Im}

\DeclareMathOperator{\Aut}{Aut}

\def\fb#1.{\underset #1 \to \times}
\def\pr#1.{\Bbb P^{#1}}
\def\ring#1.{\mathcal O_{#1}}
\def\mlist#1.#2.{{#1}_1,{#1}_2,\dots,{#1}_{#2}}

\def\Hom{\operatorname{Hom}}

\def\uloopr#1{\ar@'{@+{[0,0]+(-4,5)} @+{[0,0]+(0,10)}
@+{[0,0]+(4,5)}}
  ^{#1}}
\def\dloopr#1{\ar@'{@+{[0,0]+(-4,-5)} @+{[0,0]+(0,-10)}
@+{[0,0]+(4,-5)}}
  _{#1}}

\def\rloopd#1{\ar@'{@+{[0,0]+(5,4)} @+{[0,0]+(10,0)}
@+{[0,0]+(5,-4)}}
  ^{#1}}
\def\lloopd#1{\ar@'{@+{[0,0]+(-5,4)} @+{[0,0]+(-10,0)}
@+{[0,0]+(-5,-4)}}
  _{#1}}

\long\def\ignore#1{}
\long\def\ignore#1{#1}
\title{A lifting of an automorphism of a K3 surface over odd characteristic}
\author{Junmyeong Jang}
\date{}
\begin{document}

\newpage
 \normalsize
\maketitle
\vspace{0.3cm}
\begin{center}
{\Large
Abstract}
\end{center}
\small
In this paper, we prove that, over an algebraically closed field of odd characteristic, a weakly tame automorphism of a K3 surface of finite height can be lifted over the ring of Witt vectors of the base field. Also we prove that a non-symplectic tame automorphism of a supersingular K3 surface or a symplectic tame automorphism of a supersingular K3 surface of Artin-invariant at least 2 can be lifted over the ring of Witt vectors. Using these results, we prove, for a weakly tame K3 surface of finite height, there is a lifting over the ring of Witt vectors to which whole the automorphism group of the K3 surface can be lifted. Also we prove a K3 surface equipped with a purely non-symplectic automorphism of a certain order is unique up to isomorphism.

\vspace{0.3cm}

\normalsize
\medskip
     \section{Introduction}
          For an algebraic complex K3 surface $X$, the second integral singular cohomology $H^{2}(X,\Z)$ is a free abelian group of rank 22 equipped with a lattice structure induced by the cup product. As a lattice
     $$H^{2}(X,\Z) = U^{3} \oplus E_{8},$$
     here $U$ is a unimodular hyperbolic lattice of rank 2 and $E_{8}$ is a negative definite unimodular root lattice of rank 8.
     By the Lefschetz (1,1) theorem, the Neron-Severi group of $X$, $NS(X)$ is a primitive sublattice of $H^{2}(X,\Z)$ and
     $$NS(X) = H^{1,1}(X) \cap H^{2}(X,\Z)$$
     in $H^{2}(X,\C)$. In particular the rank of $NS(X)$ is at most 20.
     We say the rank of $NS(X)$ is the Picard number of $X$ and it is denoted by $\rho(X)$.
      $NS(X)$ is an even integral lattice of signature $(1, \rho(X)-1)$. 
     We say the orthogonal complement of the embedding
     $$NS(X) \hookrightarrow H^{2}(X,\Z)$$
     the transcendental lattice of $X$ and we denote the transcendental lattice of $X$ by $T(X)$. $T(X)$ is an integral lattice of signature
     $(2, 20-\rho(X))$.
     By the Hodge decomposition,
     $H^{0}(X, \Omega ^{2} _{X/\C})$ is a direct factor of $T(X) \otimes \C$
     and there exists a projection
     $$T(X) \otimes \C \to H^{0}(X, \Omega _{X/\C} ^{2}).$$

     By the Torelli theorem for complex K3 surfaces, an isometry $\psi \in O(H^{2}(X,\Z))$ is induced by an automorphism of $X$ if and only if $\psi$ preserves the line of holomorphic 2 forms $H^{0}(X,\Omega ^{2}_{X/\C})$ in $H^{2}(X,\Z) \otimes \C$ and the ample cone inside $NS(X) \otimes \R$.
     Let
     $$\chi _{X} : \Aut(X) \to O(T(X))$$
     and
     $$\rho _{X} : \Aut(X) \to Gl(H^{0}(X, \Omega _{X/\C} ^{2}))$$
     be the representations of the automorphism group of $X$ on the transcendental lattice and the global two forms respectively.
     Since $H^{0}(X, \Omega _{X/\C} ^{2})$ is a direct factor of $T(X) \otimes \C$, there is a canonical projection
     $$p_{X} : \im \chi _{X} \to \im \rho _{X}.$$
     It is known that $p_{X}$ is isomorphic and $\im \chi _{X}$ and $\im \rho _{X}$ are finite cyclic groups (\cite{Ni1}).
     Assume the order of $\im \rho _{X}$ is $N$ and $\xi _{N} = \rho _{X}(\alpha)$ is a primitive $N$-th root of unity. Then
     in a natural way, $T(X)$ is a free $\Z [\xi _{N}]$-module and the rank of $T(X)$ is a multiple of $\phi (N)$ (\cite{MO}).
     Here $\phi$ is the Euler $\phi$-function.\\

     An automorphism $\alpha \in \Aut (X)$ is symplectic if $\rho _{X} (\alpha) =1$. An automorphism $\alpha$ is purely non-symplectic if $\alpha$ is of finite order grater than 1 and the order of $\alpha$ is equal to the order of $\rho _{X} (\alpha)$. \\

     Assume $k$ is an algebraically closed field of odd characteristic $p$. Let $W$ be the ring of Witt vectors of $k$ and $K$ be the fraction field of $W$. Assume $X$ is a K3 surface over $k$. The formal Brauer group of $X$, $\widehat{Br} _{X}$ is a smooth one dimensional formal group over $k$ and the height of $\widehat{Br}_{X}$ is an integer between 1 and 10 or $\infty$.\\

     If the height of $X$ is $\infty$, we say $X$ is supersingular and it is known that $\rho (X) =22$ (\cite{C}, \cite{MA}, \cite{M}).
     The discriminant group of $NS(X)$, $(NS(X))^{*}/NS(X)$ is $(\Z/p) ^{2\sigma}$ for an integer $\sigma$ between 1 and 10. We call $\sigma$ the Artin-invariant of $X$. It is known that the lattice structure of $NS(X)$ is determined by the base characteristic $p$ and the Artin-invariant (\cite{RS2}). All the supersingular K3 surfaces of Artin-invariant $\sigma$ form a family of $\sigma -1 $ dimension over $k$ and a supersingular K3 surface of Artin-invariant 1 is unique up to isomorphism \cite{Og1}.\\

     If $X$ is of finite height $h$, the second crystalline cohomology has a slope decomposition (\cite{I1}, \cite{IR}, \cite{Ka})
     $$ H^{2}_{cris}(X/W) = H^{2}_{cris}(X/W) _{[1-1/h]} \oplus H^{2}_{cris} (X/W) _{[1]} \oplus H^{2}_{cris}(X/W) _{[1+1/h]}.$$
     Considering the slope spectral sequence,
      $H^{2}_{cris}(X/W) _{[1-1/h]}$ is $H^{2}(X,W\mathcal{O}_{X})$ which is isomorphic to the Dieudonn\'{e} module of $\widehat{Br}_{X}$. The Dieudonn\'{e} module of a 1 dimensional smooth formal group of finite height $h$
      can be express as
      $$W[V,F]/(VF=p, \ F= V^{h-1}).$$
       Here $F$ is a Frobenius linear operator and $V$ is a Frobenius inverse linear operator. It follows that $H^{2}_{cris}(X/W) _{[1-1/h]}$ is a free $W$-module of rank $h$. For the cup product pairing, $H^{2}_{cris}(X/W) _{[1-1/h]}$ and $H^{2}_{cris}(X/W) _{[1+1/h]}$ are dual to each other and $H^{2}_{cris}(X/W) _{[1]}$ is unimodular. Therefore the rank of $H^{2}_{cris}(X/W)_{[1]}$ is $22-2h$. Considering the cycle map
     $$ c : NS(X) \otimes W \hookrightarrow H^{2}_{cris}(X/W) _{[1]}.$$
     we have $\rho (X) \leq 22-2h$. We call the orthogonal complement of the embedding
     $$c : NS(X) \otimes W \hookrightarrow H^{2}_{cris}(X/W)$$
     the crystalline transcendental lattice of $X$ and it is denoted by $T_{cris}(X)$.
     Since
     $$H^{2}_{cris}(X/W) _{[1-1/h]} \oplus
     H^{2}_{cris}(X/W) _{[1+1/h]}$$
      is a direct factor of $T_{cris}(X)$ and there is an isomorphism
     $H^{2}(X,\mathcal{O}_{X})/V \simeq H^{2}(X,\mathcal{O}_{X})$, we have a canonical projection
     $ T_{cris}(X) \to H^{2}(X,\mathcal{O}_{X})$. We denote the representation of $\Aut (X)$ on $T_{cris}$ by
     $$\chi _{cris,X} :  \Aut(X) \to O(T_{cris}(X)).$$
     By the Serre duality, the representation of $\Aut (X)$ on $H^{2}(X,\mathcal{O}_{X})$ is isomorphic to $\rho _{X}$
     and there is a compatible projection
     $$p_{cris, X} : \im \chi _{cris,X}  \to \im \rho _{X}.$$

     For any $\alpha \in \Aut(X)$, the characteristic polynomial of $\alpha ^{*} | H^{2}_{cris}(X/W)$ has integer coefficients (\cite{I0}, 3.7.3). Hence the characteristic polynomial of $\chi _{cirs,X}(\alpha)$ also has integer coefficients. \\

     For a K3 surface $X$ of finite height over $k$, there is a Neron-Severi group preserving lifting
     $\mathfrak{X} /W$ (\cite{NO}, \cite{LM}, \cite{J}). When $X_{\bar{K}} = \mathfrak{X} \otimes \bar{K} $ is a geometric generic fiber of $\mathfrak{X}/W$, the reduction map $NS(X_{\bar{K}}) \to NS(X)$ is isomorphic and the inclusion
     $\Aut (X_{\bar{K}}) \hookrightarrow \Aut (X)$ is of finite index. Using this fact, we have that $\im \rho _{X}$ and $\im \chi _{cris,X}$ are finite (\cite{J2}). Moreover If $n$ is the order of $\chi _{cris,X}(\alpha)$, $\phi(n)$ is at most the rank of $T_{cirs}(X)$. \\

     When $X$ is a K3 surface of arbitrary height over $k$, an automorphism $\alpha \in \Aut(X)$ is tame if $\alpha$ is of finite order and the order of $\alpha$ is not divisible by the base characteristic $p$. It is known that if $p$ is greater than 11, any automorphism of finite order of $X$ is tame (\cite{DK}, Theorem 2.1.).
     If $X$ is of finite height, we say an automorphism $\alpha \in \Aut(X)$ is weakly tame if the order of $\chi _{cris,X}(\alpha)$ is not divisible by $p$. A tame automorphism is weakly tame.
     We say $X$ is weakly tame if the order of $\im \chi _{cris,X}$ is not divisible by $p$. Since the rank of $T_{cris}(X)$ is less than 22, if $p \geq 23$, any K3 surface of finite height is weakly tame.
     \\

     Let $X$ be a K3 surface over $k$.
     We say an automorphism $\alpha \in \Aut (X)$ is liftable over $W$ if there is a scheme lifting $\mathfrak{X}/W$ of $X/k$ and a $W$-automorphism $\mathfrak{a} : \mathfrak{X} \to \mathfrak{X}$ such that
     the restriction of $\mathfrak{a}$ on the special fiber $\mathfrak{a} |X$ is equal to $\alpha$. In this paper, we prove that the following theorem.\\

     \textbf{Theorem 3.3.} Let $X$ be a K3 surface over $k$. If $X$ is of finite height and $\alpha \in \Aut(X)$ is weakly tame, $\alpha$ is liftable over $W$.
      If $X$ is supersingular and $\alpha \in \Aut(X)$ is non-symplectic tame, $\alpha$ is liftable over $W$. If $X$ is supersingular of Artin-invariant at least 2 and $\alpha \in \Aut(X)$ is symplectic tame, $\alpha$ is liftable over $W$.\\

      Also, for a weakly tame K3 surface, there exists a Neron-Severi group preserving lifting which lifts all the automorphisms.\\

      \textbf{Theorem 3.7.}  Let $X$ be a weakly tame K3 surface over $k$. Then there exists a Neron-Severi group preserving lifting $\mathfrak{X}/W$ of $X$ such that
      the reduction map $\Aut (\mathfrak{X} \otimes K) \to \Aut(X)$ is isomorphic.\\

      In a previous work (\cite {J2}), we prove that, if $k$ is an algebraic closure of a finite field, $X$ is of finite height and $N$ is the order of $\im \rho _{X}$, the rank of $T_{cris}(X)$ is a multiple of $\phi (N)$. Moreover if $X$ is weakly tame,
     $p_{cris,X}$ is isomorphic. Using Theorem 3.3, we can prove the same results holds over an arbitrary algebraically closed field.\\

     \textbf{Corollary 3.5.} Let $X$ be a K3 surface of finite height over $k$. If $\alpha$ is a weakly tame automorphism of $X$ and $\rho _{X}(\alpha) =id$, then $\chi _{X}(\alpha) =id$. If $X$ is weakly tame, the projection $p_{cris,X} : \im \chi _{X} \to \im \rho_{X}$ is an isomorphism.\\

     \textbf{Corollary 3.6.}  Let $X$ be a K3 surface of finite height over $k$. When $N$ is the order of $\im \rho _{X}$, the rank of $T_{cris}(X)$ is a multiple of $\phi (N)$.\\

     Due to this result, for a weakly tame K3 surface $X$, $\im \chi _{cris,X}$ is finite cyclic.\\

     When $\Sigma$ is a finite set of positive integers $\{ 13, 17, 19, 25, 27, 32, 33, 40, 44, 50, 66 \}$, it is known that, for $N \in \Sigma$, there is a unique
     complex algebraic K3 surface $X_{N}$ equipped with a purely non-symplectic automorphism of order $N$, $g_{N}$.
        A precise elliptic surface model of $X_{N}$ is known and $X_{N}$ is defined over $\Q$. Moreover, if $p$ does not divide $2N$, $(X_{N},g_{N})$ has a good reduction over an algebraic closure of a prime field $\F_{p}$.
     It is also known that if $k$ is an algebraically closed field of characteristic $p \neq 2,3$, there is a unique K3 surface over $k$ equipped with an automorphism of order 66 (\cite{Ke2}).
          Using Theorem 3.3, we prove the uniqueness of a K3 surface equipped with a purely non-symplectic automorphism of order $N \in \Sigma$ when $p$ does not divides $2N$. This unique K3 surface is the reduction of $X_{N}$ over a finite field.\\

     \textbf{Theorem 4.3.} Assume $N \in \Sigma$ and $p$ does not divide $2N$. Then there exists a unique K3 surface equipped with a purely non-symplectic automorphism of order $N$. This unique K3 surface has a model over a finite field.\\

        \vspace{0.6cm}
    {\bf Acknowledgment}\\
               This research was supported by Basic Science Research Program through the National Research Foundation of Korea(NRF) funded by the Ministry of Education, Science and Technology(2011-0011428).\\

     \section{Deformation of a K3 surface}
     In this section we review some results on the deformation of K3 surfaces over odd characteristic. For the detail we refer to
     \cite{Og1}, \cite{De1}, \cite{De2}.\\

     Let $k$ be an algebraically closed field of odd characteristic $p$. Let $W$ be the ring of Witt vectors of $k$ and $K$ be the fraction field of $W$. Assume $X$ is a K3 surface defined over $k$.
     The deformation space of $X$ over artin $W$-algebras is an affine smooth formal scheme of 20 dimension over $W$. Let
     $B= W[[t_{1} , \cdots , t_{20}]]$ and $\mathcal{S} = \Spf B$ be the deformation space of $X$. Let $\pi : \mathcal{X} \to \mathcal{S}$ be
     the universal family over $\mathcal{S}$.
     For $A$, an artin $W$-algebra whose residue field is isomorphic to $k$,  the set of isomorphic classes of deformation of $X$ over $A$ is $\mathcal{S}(A) = \Hom _{W,cont} (B , A)$. The second derham cohomology $H=H^{2}_{dr}(\mathcal{X}/\mathcal{S})$ is a vector bundle of rank 22 on $\mathcal{S}$. The vector bundle
     $H$ is equipped wiht the Hodge filtration
     $$H= Fil ^{0} \supset Fil^{1} \supset Fil^{2} \supset 0$$
     and the Gauss-Manin connection
     $$ \nabla : H \to H \otimes _{B} \Omega ^{1}_{\mathcal{S}/W}.$$
     Here $Fil ^{1}$ and $Fil ^{2}$ are vector bundles on $\mathcal{S}$ of rank 21 and of rank 1 respectively.
     The cup product gives a perfect paring $H \otimes H \to \mathcal{O}_{\mathcal{S}}$.
     The graded module of the filtration $gr^{i} = Fil ^{i}/Fil ^{i+1}$ is a vector bundle
     and there is a natural isomorphism
     $$ gr ^{i} \simeq R^{2-i}\pi _{*} \Omega ^{i} _{\mathcal{X}/\mathcal{S}}.$$
     With respect to the cup product paring,
     \begin{center}
     $(Fil ^{1}) ^{\bot} = Fil ^{2}$ and $(Fil ^{2}) ^{\bot} =Fil ^{1}$.
     \end{center}
     By the Griffith transversality, we have
     $$ \nabla (Fil ^{2}) \subset Fil ^{1} \otimes \Omega ^{1}_{\mathcal{S}/W}.$$
      This induces an $\mathcal{O}_{\mathcal{S}}-$linear morphism
      $$ gr^{2}\nabla : gr^{2} \to gr ^{1} \otimes \Omega ^{1}_{\mathcal{S}/W}.$$
      It is known that $gr ^{2} \nabla$ is an isomorphism (\cite{De1}, Proposition 2.4.).\\

      Any $f \in \mathcal{S}(W)$ is corresponding to a formal lifting of $X$ over $\Spf W$, $\mathfrak{X} _{f} \to \Spf W$.
      There is a canonical isomorphism
      $$\lambda _{f} : f^{*}H = H^{2}_{dr}(\mathfrak{X} _{f}/W) \simeq H^{2}_{cris}(X/W).$$

      Through $\lambda _{f}$, the Hodge filtration on $H^{2}_{dr} (\mathfrak{X} _{f}/W)$,
      $$H^{2}_{dr}(\mathfrak{X}_{F}/W) \supset f^{*} Fil ^{1} \supset f^{*} Fil ^{2}$$
      gives a filtration on $H^{2}_{cris}(X/W)$. Let $M^{i}_{f} = \lambda _{f}(f^{*} Fil ^{i})$ be a sub module of $H^{2}_{cris}(X/W)$.
      A line bundle $L$ on $X$ extends on $\mathfrak{X}_{f}$ if and only if the crystalline cycle class of $L$, $c(L) \in H^{2}_{cisr}(X/W)$
      is contained in $M^{1}_{f}$ (\cite{Og1}, Proposition 1.12). The rank 1 submodule $M^{2}_{f} \subset H^{2}_{cris}(X/W)$ satisfies the following conditions.
      \begin{enumerate}
      \item $M^{2} _{f} \otimes k = H^{0}(X, \Omega _{X/k}^{2})$ through the isomorphism $H^{2}_{cris}(X/W) \otimes k \simeq H^{2}_{dr}(X/k)$.
          \item $M^{2}_{f}$ is isotropic for the cup product pairing.
          \item For the canonical Frobenius morphism $\mathbf{F} : H^{2}_{cris}(X/W) \to H^{2}_{cris}(X/W)$,
          $\mathbf{F} (M^{2}_{f}) \subset p^{2}H^{2}_{cris}(X/W)$ and $\mathbf{F} (M^{2}_{f}) \not\subset p^{3}H^{2}_{cris}(X/W)$.
      \end{enumerate}
      Let us fix a basis $v_{1}, \cdots , v_{22}$ of $H^{2}_{cris}(X/W)$ satisfying $v_{1} \in H^{2}_{f}$ and $v_{2} , \cdots , v_{21} \in H^{1}_{f}$.
      Note that $\mathbf{F}(v_{1}) \in p^{2}H^{2}_{cris}(X/W) - p^{3}H^{2}_{cris}(X/W)$,
      $\mathbf{F}(V_{i}) \in pH^{2}_{cris}(X/W) - p^{2}H^{2}_{cris}(X/W)$ for $2 \leq i \leq 21$ and
      $\mathbf{F}(v_{22}) \not\in pH^{2}_{cris}(X/W)$.
      Since the cup product pairing is perfect and the orthogonal complement of $H^{2}_{f}$ is $H^{1}_{f}$, we may assume
      $v_{1} \cdot v_{22}=1$.  Assume $M$ is a submodule of $H^{2}_{cris}(X/W)$ of rank 1 satisfying the condition 1 and the condition 2 above. There exists a unique element
      $$v_{M} = v_{1} + \sum _{i=2} ^{22} a_{i} v_{i} \in M, \ (a_{i} \in W).$$
       We can easily check that $a_{i} \in pW$
       for $2 \leq i \leq 21$, $a_{22} \in p^{2}W$ and $a_{22}$ is uniquely determined by $a_{2}, \cdots , a_{21}$. Since $\mathbf{F}(M^{1}_{f}) \subset pH^{2}_{cris}(X/W)$, the condition 3 is automatically satisfied for $M$. Let $\mathcal{M}$ be the set of rank 1 submodules of $H^{2}_{cris}(X/W)$ satisfying the condition 1 and the condition 2. The correspondence $M \mapsto (v_{2}, \cdots , v_{21})$ gives a bijection between $\mathcal{M}$ and
      $(pW) ^{20}$, so we may regard $\mathcal{M} = (pW)^{20}$.
      When $g$ is another element in $\mathcal{S}(W)$ and $\mathfrak{X}_{g}/W$ is the corresponding lifting,
      $M^{2}_{g}$ is an element of $\mathcal{M}$. Let $\Phi : \mathcal{S}(W) \to \mathcal{M}$ be the function $g \mapsto M^{2}_{g}$.

      \begin{prop}[Local Torelli theorem]\label{propo}
      The function $\Phi : \mathcal{S}(W) \to \mathcal{M}$ is bijective.
      \end{prop}
      \begin{proof}
      Let us fix a morphism $f: B \to W \in \mathcal{S}(W)$ such that $f(t_{i})=0$ for all $i$. 
      We choose $x$, a generator of  $Fil ^{2}$.
      Let us denote the differential
      $$\nabla (d/dt_{i}) : H \to H$$
      by $D_{i}$.
      Since $gr^{2}\nabla$ is isomorphic, we may choose a basis $v_{1} , \cdots , v_{22}$ of $H^{2}_{cris}(X/W)$ as above 
      such that 
      \begin{center}
      $v_{1} = \lambda _{f} (f^{*} x)$ and $v_{i} = \lambda _{f}(f^{*}D_{i}x)$ for $2 \leq i \leq 21$.
      \end{center}
      Assume $g \in \mathcal{S}(W)$ and $g(t_{i})=pa_{i} \in pW$. The $\mathcal{O}_{\mathcal{S}}$-module $H =H^{2}_{dr}(\mathcal{X}/ \mathcal{S})$ with the Gauss-Manin connection is an $F-$cyrstal in sense of $\cite{De2}$.
      Since $f \otimes k = g \otimes k$, there is an isomorphism
      $$ \chi (g,f) : H^{2}_{dr}(\mathfrak{X}_{g}/W) = g^{*}H \simeq f^{*} H = H^{2}_{dr}(\mathfrak{X}_{f}/W).$$
      Because the Gauss-Manin connection on $H$ is the connection associated to $R^{2} \pi _{cris, *} \mathcal{O}_{\mathcal{X}}$ (\cite{Be}, Proposition V 3.6.4),
      the isomorphism $\chi (f,g)$ makes the following diagram commutes
       $$ \begin{diagram}
     H^{2}_{dr}(\mathfrak{X}_{g}/W) & \rTo ^{\chi (g,f)} & H^{2}_{dr}(\mathfrak{X} _{f} /W)\\
      & \rdTo >{\lambda _{g}} &  \dTo _{\lambda _{f}} \\
     & & H^{2}_{cris}(X/W).
     \end{diagram} $$

      Precisely $\chi (g,f)$ is given as follow (\cite{De2}, Lemme 1.1.2.).
       When $m=(m_{1}, \cdots , m_{20}) \in \N ^{20}$ is a multi index,
      we denote $D^{m} = D_{1}^{m_{1}} \cdots D_{20} ^{m_{20}}$. Note that since $\nabla$ is an integrable connection, $D_{i}D_{j} = D_{j}D_{i}$ for any $i,j$. Let $\gamma _{i} : pW \to W$ be the divided power given by $\gamma _{i}(a) = a^{i}/i!$. Then
      $$\chi (g,f) (g^{*}y) = \sum _{m} \gamma_{m_{1}}(pa_{1})\cdots \gamma _{m_{20}}(pa_{20}) f^{*}(D^{m}y)$$
      for any $ y \in H$. The above summation is taken over all the multi index $m$.
      We set
      $$\lambda _{g} (g^{*} x) = \lambda _{f}(\chi (g,f)(f^{*}x)) = \sum _{i} h_{i}v_{i}.$$
      Here $h_{i} \in W[[a_{1}, \cdots , a_{20}]]$ is a formal series in $a_{i}$. In this case,
      $$h_{1} = 1 + p^{2} k_{1}$$
      and
      $$h_{i} = pa_{i} + p^{2}k_{i}\ (2 \leq i \leq 21)$$
     where all $k_{i} (1 \leq i \leq 21)$ are formal series which begins at degree 2 terms. Since $\lambda _{g} (g^{*}x)$ is a generator of $M^{2}_{g}$,
     $\Phi (g) = h_{1}^{-1} (h_{2} ,\cdots , h_{21}) \in (pW)^{20}$. By the Hensel lemma, the claim follows.     \end{proof}

     \section{Lifting of an automorphism}
     Assume $X$ is a K3 surface over $k$ and $\alpha$ is an automorphism of $X$. Let $\mathfrak{X}_{f}/W$ be the formal lifting of $X$ over $W$
     associated to $f \in \mathcal{S}(W)$.
     \begin{lem}[c.f. \cite{Og1}, Corollary 2.5.]
     An automorphism $\alpha \in \Aut (X)$ extends to $\mathfrak{X}_{f}/W$ if and only if $\alpha ^{*}| H^{2}_{cris}(X/W)$ preserves $M^{2}_{f}$.
     \end{lem}
     \begin{proof}
     The only if part is trivial. We assume $\alpha ^{*} (M^{2}_{f}) = M^{2}_{f}$. Let $\mathfrak{X}_{g} /W$ be the pull back of the lifting $\mathfrak{X} _{f}/W$ of $X/k$ through the isomorphism $\alpha$. Then there is a $W$-isomorphism $\mathfrak{a} : \mathfrak{X}_{g} \to \mathfrak{X}_{f}$ and
     we have a Cartesian diagram
      $$\begin{diagram}
     X & \rInto & \mathfrak{X}_{g}\\
     \dTo _{\alpha} & & \dTo _{\mathfrak{a}}\\
     X &  \rInto & \mathfrak{X} _{f}.
     \end{diagram}$$

     Since the isomorphism $\lambda _{f}$ and $\lambda _{g}$ are functorial, the following diagram commutes.
     $$ \begin{diagram}
     H^{2}_{dr}(\mathfrak{X}_{f}/W) & \rTo ^{\mathfrak{a}^{*}} & H^{2}_{dr}(\mathfrak{X} _{g} /W)\\
     \dTo _{\lambda _{f}} & & \dTo _{\lambda _{g}} \\
     H^{2}_{cris}(X/W) & \rTo ^{\alpha ^{*}} & H^{2}_{cris}(X/W).
     \end{diagram} $$
     Because $\mathfrak{a} ^{*} H^{0}(\mathfrak{X}_{f}, \Omega _{\mathfrak{X}_{f}/W} ^{1}) = H^{0}(\mathfrak{X}_{g},
     \Omega _{\mathfrak{X} _{g}/W} ^{2})$ and $\alpha ^{*} M^{2}_{f} = M^{2}_{f}$ by the assumption,
     $$M_{g}^{2} = \lambda _{g} (\mathfrak{a} ^{*} H^{0}(\mathfrak{X}_{f}, \Omega _{\mathfrak{X}_{f}/W} ^{1})) =
     \alpha ^{*} (\lambda _{f} (H^{0}(\mathfrak{X}_{f}, \Omega _{\mathfrak{X}_{f}/W} ^{1}))) =
     M^{2}_{f}.$$
     By Proposition \ref{propo}, $f = g$ and the automorphism $\alpha$ extends to
     an automorphism $\mathfrak{a}$ of $\mathfrak{X} _{f}$.
     \end{proof}
     \begin{rem}
     In the above lemma, if $\mathfrak{X} _{f}$ is algebrazable then $\alpha$ extends to the the algebraic model of $\mathfrak{X}_{f}$.
     \end{rem}

      \begin{thm}
      Let $X$ be a K3 surface over $k$. If $X$ is of finite height and $\alpha \in \Aut(X)$ is weakly tame, $\alpha$ is liftable over $W$.
      If $X$ is supersingular and $\alpha \in \Aut(X)$ is non-symplectic tame, $\alpha$ is liftable over $W$. If $X$ is supersingular of Artin-invariant at least 2 and $\alpha \in \Aut(X)$ is symplectic tame, $\alpha$ is liftable over $W$.

      \end{thm}
      \begin{proof}
      By Proposition 2.2 and Lemma 3.1, in each case, it is enough to find $M \in \mathcal{M}$ and an ample line bundle $V$ of $X$ such that $\alpha ^{*} M =M$ and $M$ is orthogonal to $c(V) \in H^{2}_{cris}(X/W)$.\\

      Assume $X$ is of finite height $h$ and $\alpha$ is weakly tame. We fix an $F$-crystal decomposition
      $$H^{2}_{cris}(X/W) = H^{2}_{cris} (X/W)_{[1-1/h]} \oplus H^{2}_{cris}(X/W) _{[1]} \oplus H^{2}_{cris}(X/W) _{[1+1/h]}$$
      and an identification
      $$ H^{2}_{cris}(X/W) _{[1-1/h]} = W[F,V]/(FV=p, F= V^{h-1}).$$
      Let
      $$\pi : H^{2}_{cris}(X/W) \twoheadrightarrow H^{2}_{cris}(X/W) \otimes k \simeq H^{2}_{dr}(X/k)$$
      be the canonical projection. We denote the Hodge filtration on $H^{2}_{dr}(X/k)$ by $F^{\cdot} H^{2}_{dr}(X/k)$.
      Since
      $$\mathbf{F} (H^{2}_{cris}(X/W) _{[1]} \oplus H^{2}_{cris}(X/W) _{[1+1/h]}) \subset p H^{2}_{cris}(X/W)$$
      and
      $$H^{2}_{cris}(X/W) _{[1-1/h]}/V \simeq H^{2}(X,\mathcal{O}_{X}),$$
      we have
      $$ \pi (VH^{2}_{cris}(X/W) _{[1- 1/h]} \oplus H^{2}_{cris}(X/W) _{[1]} \oplus H^{2}_{cris}(X/W) _{[1+1/h]}) = F^{1} H^{2}_{dr}(X/k).$$
      Let $v \in H^{2}_{cris}(X/W) _{[1+1/h]}$ be the dual of $1 \in W[F,V]/(FV=p, F= V^{h-1})$ with respect to the base
      $1,V, \cdots , V^{h-1}$ of $H^{2}_{cris}(X/W) _{[1-1/h]}$. Then
      $$\pi (v) \in  (F^{1}H^{2}_{dr}(X/k))^{\bot}= F^{2}H^{2}_{dr}(X/k) \subset H^{2}_{dr}(X/k)$$
       and
      $$F^{2}H^{2}_{dr}(X/k) \subset H^{2}_{cris}(X/W) _{[1+1/h]} \otimes k.$$
      \begin{lem}
      Let $L$ be a finite free $W$-module and $\psi : L \to L$ be an automorphism of $L$ of finite order coprime to $p$.
      Then there is a basis of $L$ consisting of eigenvectors for $\psi$. If $v \in L \otimes k$ is an eigenvector of $\psi | (L \otimes k)$, there is an eigenvector $\hat{v} \in L$  such that $\hat{v} \otimes k =v$.
       \end{lem}
       \begin{proof}
       Let $N$ be the order of $\psi$. Then the polynomial $t^{N} -1 \in W[t]$ splits completely.
       Therefore when $L_{\zeta}$ is the eigenspace of $(L,\psi)$ for an eigenvalue $\zeta$, we have a decomposition
       $$L = \bigoplus _{\zeta} L_{\zeta}.$$
       The claim follows easily.
       \end{proof}
            Since $\alpha$ is weakly tame and $H^{2}_{cris}(X/W)_{[1+1/h]}$ is a direct factor of $T_{cris}(X)$, the order of
            $\alpha ^{*}| H^{2}_{cris}(X/W) _{[1+1/h]}$ is not divisible by $p$.
            Because $F^{2}H^{2}_{dr}(X/k)$ is one dimensional and is invariant for $\alpha ^{*}$, it follows that, by the above lemma, there is a rank 1 $\alpha ^{*}$-stable primitive submodule $M$ of $H^{2}_{cris}(X/W) _{[1+1/h]}$ such that $\pi (M) = F^{2}H^{2}_{dr}(X/k)$. Because $H^{2}_{cris}(X/W)_{[1+1/h]}$ is isotropic for the cup product, $M $ is an element of $\mathcal{M}$. Let $f \in \mathcal{S}(W)$ be the lifting of $X$ such that $M^{2}_{f}=M$.
            Then $M ^{2}_{f} \bot H^{2}_{cris}(X/W)_{[1]}$ and $c(NS(X)) \otimes W$ is a submodule of $H^{2}_{cris}(X/W)_{[1]}$, so
            all the line bundles of $X$ extend to $\mathfrak{X}_{f}$. In particular, $\mathfrak{X}_{f}$ is algebraizable and $\alpha$ is liftable over $W$. Note that $\mathfrak{X} _{f}$ is a Neron-Severi group preserving lifting of $X$.\\

            Now assume $X$ is supersingular and $\alpha$ is non-symplectic and tame. Let
            $$H^{2}_{cris}(X/W) = \bigoplus _{\zeta} L_{\zeta}$$
            be the eigenspace decomposition for $\alpha ^{*}| H^{2}_{cris}(X/W)$.
            We assume $F^{2}H^{2}_{dr}(X/k) \subset \pi (L_{\zeta _{0}})$ for some eigenvalue $\zeta _{0} \neq 1$. Then $\rho _{X}(\alpha) = \bar{\zeta _{0}}$, where
            $\bar{\zeta _{0}}$ is the reduction of $\zeta _{0}$ in $k$. If $\zeta _{0}\neq -1$, $L_{\zeta _{0}}$ is isotropic and there is a rank 1 primitive submodule $M \subset L_{\zeta _{0}}$ satisfying $\pi(M) = F^{2}H^{2}_{dr}(X/k)$. Then $M$ is an element of $\mathcal{M}$. If $\zeta _{0} =-1$, $\rho _{X}(\alpha) =-1$ and
            by the Serre duality, $\alpha ^{*} | (H^{2}_{dr}(X/k)/F^{1}) =-1$.
            We set $l_{-1} = \pi (L_{-1})$. The pairing on $l_{-1}$ is non-degenerate.
            The rank of $L_{-1}$ is at least 2 and
            $$l_{-1} \not\subset F^{1}H^{2}_{dr}(X/k) = (F^{2}H^{2}_{dr}(X/k))^{\bot}.$$
            Let us choose $0 \neq x \in F^{2}H^{2}_{dr}(X/k)$ and $y \in l_{-1}$ such that $ x \cdot y =1$. Let $u$ and $v$ are liftings of $x$ and $y$ in $L_{-1}$ satisfying
            $u \cdot v =1$. Since $u \cdot u$ is divisible by $p$, by the Hensel lemma, there is $a \in W$ such that $u + pav \in L_{-1}$ is isotropic.
            If $M$ is a submodule of $L_{-1}$ generated by $u+pav$, $M$ is an element of $\mathcal{M}$.
           Let $f \in \mathcal{S}(W)$ be the formal lifting of $X$ over $W$ such that $M = M_{f}^{2}$.
            Because $\alpha$ is of finite order, there an $\alpha ^{*}$-stable ample line bundle of $X$, $V$. Then $c(V) \in L_{1}$ and $c(V) \bot L_{\zeta _{0}} \supset M$. Therefore the formal lifting $\mathfrak{X} _{f}$ is algebraizable and $\alpha$ is liftable over $W$.    \\

            Assume $X$ is supersingular of Artin-invariant at least 2 and $\alpha$ is symplectic and tame.
            We set $l_{1} =\pi (L_{1})$. The pairing on $l_{1}$ is non-degenerate.
            By the assumption, $F^{2}H^{2}_{dr}(X/k) \subset l_{1}$. Let $V$ be a primitive $\alpha ^{*}$-ample bundle. Then $c(V) \in L_{1}$
            and $\pi (c(V)) \neq 0$. Let $x$ be a non-zero element of $F^{2}H^{2}_{dr}(X/k)$ and $y = \pi (c(V))$.
            By \cite{Og1}, Proposition 2.2,
            $x$ and $y$ are linearly independent.
            We denote the kernel of $l_{1} \twoheadrightarrow H^{2}(X,\mathcal{O}_{X})$ by $F^{1}l_{1}$. $F^{1}l_{1}$ is of codimension 1 in $l_{1}$. Note that $x,y \in F^{1}l_{1}$ and $F^{1}l_{1}$ is the orthogonal complement of $x$ in $l_{1}$.
            Let $c(V)^{\bot}$ be the orthogonal complement of $c(V)$ in $L_{1}$.
            Suppose the self intersection of $c(V)$ is not divisible by $p$.
            Since $x \cdot y =0$,
            $$F^{2}H^{2}_{dr}(X/k) \subset \pi (c(V)^{\bot})$$
             and
            $$\pi (c(V) ^{\bot}) \not\subset F^{1}H^{2}_{dr}(X/k).$$
            Therefore the rank of $c(V) ^{\bot}$ is at least 2 and as above there a rank 1 submodule $M$ of $c(V)^{\bot}$ such that $\pi(M) =F^{2}H^{2}_{dr}(X/k)$.
            The formal lifting corresponding to $M$ is algebrazable and
            $\alpha$ is liftable to the scheme lifting corresponding to $M$.
            Suppose the self intersection of $c(K)$ is divisible by $p$. Then $y$ is isotropic. Since $x$ and $y$ are linearly independent, there is $y \neq z \in F^{1}l_{1}$ such that $z \cdot y =1$. Hence the dimension of $F^{1}l_{1}$ is at least 3 and the rank of $L_{1}$ is at least 4. Let $v$ and $u$ are arbitrary liftings of $x$ and $z$ in $L_{1}$ respectively. Then $v \cdot u$ is divisible by $p$ and $c(K) \cdot u$ is a unit. We choose $w \in L_{1}$ such that $v \cdot w$ is a unit. We can find $a,b \in W$ satisfying
            $$v + a u , w + bu \in (c(K))^{\bot}.$$
             Since $v \cdot c(K)$ is divisible by $p$, $a \in pW$ and $(v+au) \cdot (w+bu)$ is a unit. Then inside $(c(K))^{\bot}$, we can find a rank 1 isotropic submodule $M$ such that $\pi(M) = F^{2}H^{2}_{dr}(X/k)$. The formal lifting associated to $M$ is algebraizable and $\alpha$ is liftable to the scheme lifting associated to $M$.
      \end{proof}

      \begin{cor}Let $X$ be a K3 surface of finite height over $k$. If $\alpha$ is a weakly tame automorphism of $X$ and $\rho _{X}(\alpha) =id$, then $\chi _{X}(\alpha) =id$. If $X$ is weakly tame, the projection $p_{cris,X} : \im \chi _{X} \to \im \rho_{X}$ is an isomorphism.
      \end{cor}
      \begin{proof}
      Since $\alpha$ is weakly tame, as in the proof of the above theorem, there is a Neron-Severi group preserving lifting $\mathfrak{X}/W$ of $X$ equipped with an automorphism $\mathfrak{a} : \mathfrak{X} \to \mathfrak{X}$ satisfying $\mathfrak{a} \otimes k =\alpha$. Let $X_{K}/K$ be the generic fiber of $\mathfrak{X}/W$. By the assumption $\mathfrak{a} ^{*} | H^{0}(X_{K}, \Omega _{X_{K}/K}) = id$. Since $K$ is of characteristic 0, $\mathfrak{a} ^{*}| T(X_{K}) =id$. But there is a functorial isomorphism
      $$ H^{2}_{dr}(X_{K}/K) \simeq H^{2}_{cris}(X/W) \otimes K,$$
      so $\alpha ^{*}| T_{cris}(X) =id$. The later part follows easily.
      \end{proof}

      \begin{cor}
      Let $X$ be a K3 surface of finite height over $k$. When $N$ is the order of $\im \rho _{X}$, the rank of $T_{cris}(X)$ is a multiple of $\phi (N)$.
      \end{cor}
      \begin{proof}
      Let $\alpha$ be an automorphism of $X$ such that $\rho _{X}(\alpha)$ generates $\im \rho _{X}$.
      We assume the order of $\chi _{cris,X} (\alpha)$ is $p^{r}M$ where $M$ is a positive integer which is not divisible by $p$.
      Then $\alpha ^{p^{r}}$ is weakly tame and $M$ is equal to $N$ by the above corollary. Replacing $\alpha$ by $\alpha ^{p^{r}}$, we may assume $\alpha$ is weakly tame. Then there exist a Neron-Severi group preserving lifting $\mathfrak{X}/W$ and a lifting of $\alpha$,
      $\mathfrak{a} : \mathfrak{X} \to \mathfrak{X}$. Since the order of $\rho _{X_{K}} (\mathfrak{a})$ is $N$, the rank of $T(X_{K})$ is a multiple of $\phi (N)$. The rank of $T(X_{K})$ is equal to the rank of $T_{cris}(X)$ and the claim follows.

      \end{proof}

      \begin{thm}
      Let $X$ be a weakly tame K3 surface over $k$. There exists a Neron-Severi group preserving lifting $\mathfrak{X}/W$ of $X$ such that
      the reduction map $\Aut (\mathfrak{X} \otimes K) \to \Aut(X)$ is isomorphic.
      \end{thm}
      \begin{proof}
      Let $h$ be the height of $X$.
      Let $\alpha$ be a weakly tame automorphism of $X$ such that $\rho _{X}(\alpha)$ generates $\im \rho _{X}$. Then by Corollary 3.5,
      $\chi _{cris,X}(\alpha)$ generates $\im \chi _{cris,X}$. As in the proof of Theorem 3.3, we can find $M \in \mathcal{M}$ inside $H^{2}_{cris}(X/W) _{[1+1/h]}$ which is $\alpha ^{*}$-stable. Let $\mathfrak{X}/W$ be the lifting of $X$ corresponding to $M$.
      Then $\alpha$ is liftable to $\mathfrak{X}$. For any $\beta \in \Aut (X)$, $\chi _{cris ,X} (\beta) = \chi _{cris,X} (\alpha ^{i})$ for some integer $i$. Since $$H^{2}_{cris}(X/W) _{[1+1/h]} \subset T_{cris}(X),$$
       $M$ is stable for $\beta ^{*}$ and $\beta$ is liftable to $\mathfrak{X}$. Therefore
       $$\Aut (\mathfrak{X} \otimes K) = \Aut (\mathfrak{X}) \to \Aut(X)$$
       is surjective.

      \end{proof}

      \begin{rem}
      Assume $p$ is at least 5 and $X$ is a supersingular K3 surface of Artin-invariant 1 over $k$. Then $\im \rho _{X}$ is a cyclic group of order $p+1$ (\cite{J3}). Hence if $p>60$, $\phi (p+1)>21$ and there is an automorphism of $X$ which can not be lifted over characteristic 0. It is also known that for a supersingular K3 surface of Artin-invariant 1 over a field of characteristic 3, there is an automorphism which can not be lifted over characteristic 0 (\cite{EO}).
      We can ask whether for any supersingular K3 surface, there is an automorphism which can not lifted over characteristic 0.
      \end{rem}

      \section{Non-symplectic automorphisms}
      Let $k$ be an algebraically closed field of odd characteristic $p$ whose cardinality is equal to or less than the cardinality of the real numbers. Let $W$ be the ring of Witt-vectors of $k$ and $K$ be the fraction field of $W$. Let $\bar{K}$ be an algebraic closure of $K$. We fix an isomorphism $\bar{K} \simeq \C$. Let $\Sigma = \{13,17,19,25,27, 32,33,40,44,50,66 \}$ be a finite set of positive integers. The following is known.

      \begin{thm}[\cite{Ko}, \cite{MO} , \cite{OZ}, \cite{Ta}] If $N \in \Sigma$, there exists a unique complex algebraic K3 surface $X_{N}$ equipped with a purely non-symplectic automorphism of order $N$, $g_{N}$ up to isomorphism. $X_{N}$ has a model over $\Q$ and if a prime number $p$ does not divide $2N$, $(X_{N},g_{N})$ has a good reduction $(X_{N,p},g_{N,p})$ over an algebraic closure of a prime field of characteristic $p$.

      \end{thm}
     In the case of $N=66$, the following result over positive characteristic is also known.
     \begin{thm}[\cite{Ke2}] If the characteristic of $k$ is not 2 or 3, there is a unique K3 surface equipped with an automorphism of order 66.
     \end{thm}
     Note that the above result covers a wild case of characteristic 11.\\

     Using Theorem 3.3, we prove the uniqueness of a K3 surface over $k$ equipped with a purely non-symplectic tame automorphism of order $N$ for $N \in \Sigma$.

     \begin{thm}
     Assume $N \in \Sigma$ and $p$ does not divide $2N$. Then there exists a unique K3 surface equipped with a purely non-symplectic automorphism of order $N$ up to isomorphism. This unique K3 surface has a model over a finite field.
     \end{thm}
     \begin{proof}
     The existence is guaranteed by Theorem 4.1. Now assume $X$ is a K3 surface over $k$ and $\alpha \in \Aut (X)$ is purely non-symnplectic of order $N$. Since $\alpha$ is non-symplectic tame, by Theorem 3.3, there exists a scheme lifting $\mathfrak{X}/W$ of $X$ and
     an automorphism $\mathfrak{a} : \mathfrak{X} \to \mathfrak{X}$ such that $\mathfrak{a} \otimes k = \alpha$. Then $\mathfrak{X} \otimes \C$ is a complex K3 surface equipped with a purely non-symplectic automorphism of order $N$, so $\mathfrak{X} \otimes \C \simeq X_{N}$.
     It follows that $X$ is isomorphic to $X_{N,p} \otimes k$.
     \end{proof}

\vskip 1cm

\noindent
J.Jang\\
Department of Mathematics\\
University of Ulsan \\
Daehakro 93, Namgu Ulsan 680-749, Korea\\ \\
jmjang@ulsan.ac.kr

     \end{document}